\documentclass[leqno]{amsart}
\usepackage{amsmath}
\usepackage{amssymb}
\usepackage{amsthm}
\usepackage{enumerate}
\usepackage[mathscr]{eucal}
\theoremstyle{plain}
\newtheorem{theorem}{Theorem}[section]
\newtheorem{prop}[theorem]{Proposition}

\newtheorem{lemma}{Lemma}[section]

\theoremstyle{definition}

\newtheorem{remark}{Remark}[section]

\begin{document}

\newcommand{\acr}{\newline\indent}
\title[Symmetry of Birkhoff-James orthogonality of operators ]{Symmetry of Birkhoff-James orthogonality of operators defined between infinite dimensional Banach spaces}
\author[  Kallol Paul, Arpita Mal and Pawel W\'ojcik]{ Kallol Paul, Arpita Mal and Pawel W\'ojcik}

\address[Paul]{Department of Mathematics\\ Jadavpur University\\ Kolkata 700032\\ West Bengal\\ INDIA}
\email{kalloldada@gmail.com}

\address[Mal]{Department of Mathematics\\ Jadavpur University\\ Kolkata 700032\\ West Bengal\\ INDIA}
\email{arpitamalju@gmail.com}

\address[W\'ojcik]{Institute of Mathematics\\ Pedagogical University of Cracow \\
Podchor\c a\.zych 2, 30-084 Krak\'ow\\ Poland }
\email{ pawel.wojcik@up.krakow.pl}

\thanks{First author acknowledges the generosity of Pedagogical University of Krak\'ow, Poland and in particular, Professor Jacek Chmieli\'nski,   for supporting the visit to the university during April 2018. This research paper originated from that visit. Second author would like to thank UGC, Govt. of India for the financial support.
}

\subjclass[2010]{Primary 47L05, Secondary 46B20}
\keywords{Birkhoff-JamesOrthogonality; symmetric operators; infinite dimensional Banach space}

\begin{abstract}
We study left symmetric bounded linear operators in the sense of Birkhoff-James orthogonality defined between infinite dimensional Banach spaces. We prove that a bounded linear operator defined between two strictly convex Banach spaces is left symmetric if and only if it is zero operator when the domain space is reflexive and Kadets-Klee. We exhibit a non-zero left symmetric operator when the spaces are not strictly convex.
We also study right symmetric bounded linear operators between infinite dimensional Banach spaces.

\end{abstract}

\maketitle

\section{Introduction}
The study of left symmetric and right symmetric operators in the sense of Birkhoff-James orthogonality is an interesting area of research in the space of bounded linear operators between Banach spaces.  Turn\v sek \cite{Ta} studied such operators when the underlying space is a Hilbert space. The characterization of  left symmetric and right symmetric operators between Banach spaces is still an open problem. Recently, Sain et. al. \cite{SGP} studied those operators between finite dimensional Banach spaces. This paper is based on the study of left symmetric and right symmetric operators between infinite dimensional Banach spaces. Before proceeding further, we fix the notations and terminologies.

\smallskip

Let $\mathbb{X},\mathbb{Y}$ denote real normed linear spaces with  dim $\mathbb{X} > 1$ and  dim $\mathbb{Y} > 1$, unless otherwise mentioned. Let $B_{\mathbb{X}}=\{x\in \mathbb{X}:\|x\|\leq 1\}$ and $S_{\mathbb{X}}=\{x\in \mathbb{X}:\|x\|= 1\}$ denote the unit ball and the unit sphere of $\mathbb{X}$ respectively. Let $B(\mathbb{X},\mathbb{Y})~(K(\mathbb{X},\mathbb{Y}))$ denote the space of all bounded (compact) linear operators between $\mathbb{X}$ and $\mathbb{Y}.$ For any two elements $x,y\in \mathbb{X},$  $x$ is said to be Birkhoff-James orthogonal \cite{B,J} to $y,$ written as $x\perp_B y,$ if $\|x+\lambda y\|\geq \|x\|$ for all scalar $\lambda.$ It is easy to see that Birkhoff-James orthogonality notion is, in general, not symmetric. James \cite{Ja} proved that if dim $\mathbb{X}\geq 3$ and Birkhoff-James orthogonality is symmetric then the norm is induced by an inner product. An element $x\in \mathbb{X}$ is said to be left symmetric (right symmetric)  if for any element $ y \in \mathbb{X},~ x\perp_B y  \Rightarrow y\perp_B x$ ($y\perp_B x \Rightarrow x\perp_B y$ ).  A normed linear space $\mathbb{X}$ is said to be strictly convex if the unit sphere does not contain any straight line segment, i.e., for any  $x,y\in S_{\mathbb{X}},~\|(1-t)x+ty\|=1$ for some $ t \in (0,1) $  implies $ x=y.$   An element $x\in S_{\mathbb{X}}$ is said to be a smooth point if $x$ has a unique norming linear functional, i.e., there exists a unique $f\in S_{\mathbb{X}^*}$ such that $f(x)=1.$ A normed linear space $\mathbb{X}$ is said to be smooth if every element of $S_{\mathbb{X}}$ is smooth.
Sain \cite{S} introduced the notion of $x^{+}, x^{-}$ in studying Birkhoff-James orthogonality which are defined as follows: For any two elements $ x, y\in \mathbb{X}, $  $ y \in x^{+} $ if $ \| x + \lambda y \| \geq \| x \| $ for all $ \lambda \geq 0. $ Similarly, we say that $ y \in x^{-} $ if $ \| x + \lambda y \| \geq \| x \| $ for all $ \lambda \leq 0. $ The notion was further generalized in \cite{SPM} as follows: For $x,y\in \mathbb{X}$ and $\epsilon \in [0,1),$   $y\in x^{+\epsilon}$ if $\|x+\lambda y\|\geq \sqrt{1-\epsilon^2}\|x\|$ for all $\lambda \geq 0.$ Similarly, we say that $y\in x^{-\epsilon}$ if $\|x+\lambda y\|\geq \sqrt{1-\epsilon^2}\|x\|$ for all $\lambda \leq 0.$  The norm attaining set $M_T$ of $T$ plays an important role in our study which is defined as : For $T\in B(\mathbb{X},\mathbb{Y}),$  $ M_T=\{x\in S_{\mathbb{X}}:\|Tx\|=\|T\|\},$ i.e., it is the set of all unit vectors at which $T$ attains its norm. A sequence $\{x_n\}$ of unit vectors is said to be a norming sequence for $T$ if $ \|Tx_n\| \rightarrow \|T\|.$

\smallskip

It is easy to observe that  $B(\mathbb{X},\mathbb{Y})$ is neither strictly convex nor smooth and so it is not an inner product space. Thus, it becomes interesting to find out the elements in $ B(\mathbb{X},\mathbb{Y}) $ which are left symmetric and right symmetric. Turn\v sek \cite{Ta} proved that $T\in B(\mathbb{H},\mathbb{H})$ is left symmetric if and only if $T=0.$ Sain et. al. \cite{SGP} studied left symmetric operators in $B(\mathbb{X},\mathbb{X})$ when dim $\mathbb{X}$ is finite, $\mathbb{X}$ is strictly convex and smooth. In second section, we characterize left symmetric operators in $B(\mathbb{X},\mathbb{Y})$, with  $\mathbb{X},\mathbb{Y}$  not necessarily finite dimensional and also not smooth.   We prove that $T\in B(\mathbb{X},\mathbb{Y})$ with $M_T\neq \emptyset$ is left symmetric if and only if $T=0$, where, $\mathbb{X}$ is reflexive, Kadets-Klee and strictly convex Banach space and $Y$ is any strictly convex Banach space. We also prove that $T\in K(\mathbb{X},\mathbb{Y})$ is left symmetric if and only if $T=0,$ where $\mathbb{X}$ is reflexive, strictly convex Banach space and $\mathbb{Y}$ is any strictly convex Banach space. Next, we seek for non-zero left symmetric operators and obtain some positive results in this direction.  We exhibit non-zero left symmetric operators defined between $\mathbb{X}\oplus_1 \mathbb{R}$ and $\mathbb{Y},$ where $\mathbb{X}$ is a reflexive Banach space, $\mathbb{Y}$ is reflexive smooth Banach space. Note that if $\mathbb{X},~\mathbb{Z}$ are normed linear spaces, then $\mathbb{X}\oplus_1 \mathbb{Z}$ denote the space $\mathbb{X}\times \mathbb{Z}$ with $\|(x,z)\|_1:=\|x\|+\|z\|.$ We further characterize left symmetric compact operators from $\ell_1^n$ to a reflexive smooth Banach space.

In third section, we study right symmetric operators. Turn\v sek \cite{Ta} and Ghosh et. al. \cite{GSP} studied independently right symmetric operators between Hilbert spaces. In the case of  $\mathbb{X}$ being a Banach space, not necessarily a Hilbert space, the study of right symmetric operators in $B(\mathbb{X},\mathbb{X})$ becomes  more involved. In \cite{SGP} Sain et. al. studied right symmetric operators in  $B(\mathbb{X},\mathbb{X}),$ where $\mathbb{X}$ is finite dimensional. Here we study right symmetric operators between infinite dimensional Banach spaces which substantially improves on
results of \cite{SGP}.

\section{Left symmetric operators}

We begin this section with an easy proposition which explores the connection between symmetricity, smoothness and strict convexity in a normed linear space.
\begin{prop}\label{prop-symm}
Let $ \mathbb{X} $ be a normed linear space. \\
(i) If $ x \in S_\mathbb{X}$ is right symmetric and smooth then $x$ is left symmetric. \\
(ii) If $ \mathbb{X} $ is strictly convex and $ x \in S_\mathbb{X}$ is left symmetric  then $x$ is right symmetric. \\
(iii) If $ \mathbb{X} $ is strictly convex and $ x \in S_\mathbb{X}$ is left symmetric  then $x$ is smooth.
\end{prop}
\begin{proof} (i) Let $ x$ be right symmetric and smooth. Let $ x \bot_B y.$ Then by \cite[Th. 2.3]{J}, there exists $a \in \mathbb{R} $ such that $ (ax+y) \bot_B x.$ Since $x$ is right symmetric so $ x \bot_B (ax+y).$ Again $x$ being smooth, by \cite[Th. 4.1]{J}, there exists unique $a\in \mathbb{R}$ such that $ x \bot_B (ax +y).$ So we must have $a=0$ and hence $ y \bot_B x.$ Thus $x $ is left symmetric. \\
(ii)  Let $ y \bot_B x.$ Then by \cite[Cor. 2.2]{J}, there exists $b\in \mathbb{R}$ such that $x \bot_B (bx +y).$ Since $x$ is left symmetric so $ (bx+y) \bot_B x.$ Being strictly convex, by \cite[Th. 4.3]{J}, there exists unique scalar $b$ such that $ (bx+y) \bot_B x.$ So we must have $b=0$ and hence $ x \bot_B y.$ Thus, $x$ is right symmetric. \\
(iii)  If possible, let $x$ be not smooth. Then Birkhoff-James orthogonality at $x$ is not right unique. Therefore, there exists $ y \in  \mathbb{X}$ and distinct scalars $a$ and $b$ such that $ x \bot_B (ax+y ) $ and $ x \bot_B (bx+y).$ Since $x$ is left symmetric, so $ (ax+y) \bot_B x$ and $ (bx+y) \bot_B x.$ Thus, Birkhoff-James orthogonality is not left unique. This contradicts the fact that $\mathbb{X}$ is strictly convex. Hence, $ x$ is smooth.
\end{proof}

One of the main tool that is being used to study left symmetric operators is the connection between orthogonality in the space of linear operators and that in the ground space. Paul et. al. \cite{PSG} and Sain et. al. \cite{SPM} explored this connection for  compact linear operators. In the same line of thinking we prove the following theorem.
\begin{theorem}\label{th-Kadets-Klee}
	(i) Let $\mathbb{X}$ be a reflexive Kadets-Klee Banach space, $\mathbb{Y}$ be  any Banach space and $ T \in K(\mathbb{X}, \mathbb{Y})$.  Then for any $ A \in B(\mathbb{X}, \mathbb{Y})$, $ T \bot_B A$ if and only if there exists $x,y \in M_T$ such that $ Ax \in (Tx)^+$ and $ Ay \in (Ty)^-.$ \\
	(ii) In addition, if $M_T = D \cup (-D), $ where $D$ is a compact, connected subset of $ S_\mathbb{X}$, then there exists $x \in M_T$ such that $Tx \bot_B Ax.$
\end{theorem}
\begin{proof}
	$(i)$ The sufficient part follows trivially. We only prove the necessary part. Let $T \bot_B A.$ We first claim that for any norming sequence $\{x_n\}$  for $T$ there exists $ x \in M_T$ such that $ x_{n_k} \rightarrow x$ for some subsequence $\{x_{n_k}\}.$ Since $\mathbb{X}$ is reflexive, $B_{\mathbb{X}}$ is weakly compact. Therefore, there exists a  subsequence $\{x_{n_k}\}$ which is weakly convergent to some element $x \in B_{\mathbb{X}}.$ Since $T$ is compact, so $ Tx_{n_k} \rightarrow Tx.$ Again, $\{x_n \}$ is a norming sequence for $T.$ Thus, $\|Tx\| = \|T\| $ and hence $ x \in M_T.$ Now, $ \|x_{n_k}\| = \|x\| = 1. $ Since $x_{n_k}\rightharpoonup x,~\|x_{n_k}\|\to \|x\|$ and $ \mathbb{X}$ is Kadets-Klee, so $ x_{n_k} \rightarrow x.$ This justifies our claim. Since $T \bot_B A,$ so by \cite[Th. 2.4]{SPM} either $(a)$ or $(b)$ holds: \\
	$(a)$ There exists a  norming sequence $\{x_n\}$ for $T$ such that  $\|Ax_n\| \rightarrow 0$ as $n\rightarrow \infty$. \\
	$(b)$ There exists two norming sequences $\{x_n\},~\{y_n\}$ for $T$  and two sequences of positive real numbers $\{\epsilon_n\}$ , $\{\delta_n\}$ such that $  \epsilon_n \rightarrow 0$, $\delta_n\rightarrow 0$ as $n\rightarrow \infty$ and
	$ Ax_n\in(Tx_n)^{+(\epsilon_n)}$ and $Ay_n\in(Ty_n)^{-(\delta_n)}$ for all  $n\in \mathbb{N}$. \\
	If $(a)$ holds, then by the above claim we get $ x \in M_T$ such that $Ax = 0$ and so $Tx \bot_B Ax.$  If $(b)$ holds, then  $ \| Tx_n + \lambda Ax_n \| \geq \sqrt{1 - \epsilon^2} \|Tx_n\| $ for all $\lambda \geq 0$ and for all $n \in \mathbb{N}.$ Again following our claim we can find $ x \in M_T$ such that  $ \| Tx + \lambda Ax \| \geq \|Tx\|$ for all $\lambda \geq 0$, i.e., $Ax \in (Tx)^+.$ Similarly, we get $ y \in M_T$ such that $ Ay \in (Ty)^-.$
	This completes the proof.
	
	\smallskip
	$(ii)$ Suppose there does not exist any $x\in M_T $ such that $Tx \bot_B Ax.$ Then from \cite[Lemma 2.1]{PSG}, there exists $ \lambda_0 \neq 0 $ such that $ \| Tx + \lambda_0Ax\| < \|Tx\|~ \forall~ x \in M_T.$ Let $\lambda_0>0.$Then $Ax \notin (Tx)^+$ for all $x\in M_T.$ This contradicts $(i).$ Similarly, $\lambda_0<0$ implies that $Ax \notin (Tx)^-$ for all $x\in M_T.$ This again contradicts $(i).$ Therefore, there exists $x\in M_T$ such that $Tx\perp_B Ax.$ This completes the proof of the theorem.
\end{proof}

Using \cite[Th. 2.1]{PSG} we now prove the following theorem which improves on \cite[Th. 2.5]{S}.

\begin{theorem}\label{th-1111}
	Let $\mathbb{X},\mathbb{Y}$ be normed linear spaces.
	Suppose that
	$\mathbb{X}$ is strictly convex and reflexive.
	Suppose that $T\in K(\mathbb{X},\mathbb{Y})$ is left symmetric with $\|T\|=1$. If
	$x_1\in M_T$, $y\in S_\mathbb{X}$, $y\perp_B x_1$, then $Ty=0$.
\end{theorem}
\begin{proof}
	Assume, contrary to
	our claim, that $Ty\neq 0$. Since $y\perp_B x_1$, there is $f\in \mathbb{X}^*$ such
	that $\|f\|=1$, $f(y)=1$, $f(x_1)=0$.
	Define $A\in K(\mathbb{X},\mathbb{Y})$ by $Ax:=f(x)Ty$.
	Since $Tx_1\perp_B Ax_1$ and $x_1\in M_T$, we have $T\perp_B A$.
	It follows that $A\perp_B T$. Since $\mathbb{X}$ is strictly convex, $M_A=\{y,-y\}$.
	Thus, from \cite[Th. 2.1]{PSG} it follows that $Ay\perp_B Ty$, so
	$Ty\perp_B Ty$.
	Therefore $Ty=0$, a contradiction. This completes the proof of the theorem.
\end{proof}

Similarly, using Theorem \ref{th-Kadets-Klee}, we can easily prove the following theorem.
\begin{theorem}\label{th-rank1}
	Let $\mathbb{X}$ be a Banach space which is reflexive, Kadets-Klee, strictly convex and $\mathbb{Y}$ be a Banach space. Let $T \in B(\mathbb{X}, \mathbb{Y})$ be left symmetric  with $M_T \neq \emptyset.$ Then for
	$x\in M_T$ if $y\perp_B x$, then $Ty=0$.
\end{theorem}

Next, we prove the lemma, which is needed in our main result.
\begin{lemma}\label{lemma:minus}
	Let $\mathbb{X}$ be a strictly convex normed linear space and $ u, v \in \mathbb{X}.$ Let $ v \bot_B u$ and $ w = (1-t)u + tv,$ where $ t \in (0,1).$ Then $ au \not\in (av+bw)^{-},$ if $ab >0.$
\end{lemma}
\begin{proof}
	Let $a>0, b >0.$ Then $ av + bw = av + b(1-t)u + btv = (a+bt)v + b(1-t)u.$ Clearly, $ a+bt >0$ and $ b(1-t) > 0.$ Now, since $\mathbb{X}$ is strictly convex and $ v \bot_B u,$ $\| (a+bt)v\| < \| (a+bt)v + b(1-t)u\|.$ Thus,
	$ \| (av+bw) - b(1-t)u \| = \| (a+bt)v\| < \| (a+bt)v + b(1-t)u\| = \| av+bw\|.$
	This shows that $ u \not\in (av+bw)^{-}.$ Hence, by Proposition \cite[2.2]{Sa}, $au\not\in (av+bw)^{-},$ since  $a>0.$ \\
	Next, consider $a<0,b<0.$ Then as previous, $ \| (av+bw) - b(1-t)u \| = \| (a+bt)v\| < \| (a+bt)v + b(1-t)u\| = \| av+bw\|.$ Here $-b(1-t)>0.$ This shows that $ u \not\in (av+bw)^{+}.$ Again, by Proposition \cite[2.2]{Sa}, $au\not\in (av+bw)^{-},$ since  $a<0.$ This completes the proof of the lemma.\\
\end{proof}

Now, we are in a position to prove our main result.
\begin{theorem}\label{Th-t=0}
	Let $\mathbb{X}$ be a Banach space which is   reflexive, Kadets-Klee, strictly convex and $\mathbb{Y}$ be a strictly convex Banach space. Then $T \in B(\mathbb{X}, \mathbb{Y})$  with $M_T \neq \emptyset$ is left symmetric if and only if $T$ is the zero operator.
\end{theorem}
\begin{proof} The sufficient part is trivial. We prove the necessary part  in three steps by the method of contradiction. Let $T$ be left symmetric but $T$ is non-zero.\\
	\textbf{Step 1.} We show that for each $x \in M_T$ there exists a hyperspace $H$ such that $x \bot_B H$  and $T(H) = 0.$  \\
	Let $ x \in M_T.$ Then from Theorem \ref{th-rank1} it follows that $Ty = 0$ for all $y$ with $ y \bot_B x.$ Since $ x \in M_T$ so by \cite[Lemma 2.1]{SPM2}, there exists a hyperspace $H$  such that $ x \bot_B H $ and $ Tx \bot_B T(H).$ We claim that $T(H) =0.$ Let $ y \in H \cap S_{\mathbb{X}}.$ Then there exists  a hyperspace $H_y$ such that $ y \bot_B H_y.$  Define a linear operator $ A : \mathbb{X} \longrightarrow \mathbb{Y}$ as : $ A(cy+h) = cTy$, where $ c$ is a scalar and $ h \in H_y.$ Clearly $A$ is compact and $M_A= \{ \pm y\} $, since $\mathbb{X}$ is strictly convex. Let $ x = by + h .$ Then $ Ax = b Ty.$  Since $ x \bot_B y,$ $ Tx \bot_B Ty $. This shows that $ Tx \bot_B Ax.$ Thus, $ T \bot_B A.$ Since $T$ is left symmetric, we get $ A \bot_B T.$ Hence, by Theorem \ref{th-Kadets-Klee} (ii), $ Ay \bot_B Ty. $ Thus, $ Ty \bot_B Ty.$ This forces $Ty=0.$ Thus, our claim $ T(H) = 0$  is established. \\
	\textbf{Step 2.} We show that $ H \bot_B x$ and $Tx$ is left symmetric. \\
	
	Let $ h \in H$. Then $ Th=0.$ There exists $ d\in \mathbb{R} $ such that $ dx + h \bot_B x.$ So $ T(dx+h) =0 $ implies $ d=0.$ Thus $ h \bot_B x$ and so $ H \bot_B x.$\\
	Next, we prove that $Tx$ is left symmetric. Let $Tx \bot_B u.$ As before, defining $ A : \mathbb{X} \longrightarrow \mathbb{Y}$ by : $ A(ax+h) = au$, where $ a$ is a scalar and $ h \in H$, we see that $M_A = \{ \pm x \}.$ Since $Tx\perp_B Ax$ and $x\in M_T,$ $ T \bot_B A.$ It follows that $A \bot_B T,$ since $T$ is left symmetric. Therefore, by Theorem \ref{th-Kadets-Klee} (ii), $Ax \bot_B Tx$, i.e., $ u \bot_B Tx.$ This proves that $Tx$ is left symmetric. \\
	
	\textbf{Step 3.}  We construct an operator $A$ such that $ T \bot_B A$ but $ A\not\perp_B T.$ \\
	
	Consider a unit vector $y \in H $. Clearly $ x \bot_B y$ and $ y \bot_B x.$  Then $y \bot_B H_y$ where $H_y$ is a subspace of $H$ of codimension one in $H.$ Let $v \in S_{\mathbb{X}}$ such that $Tx \bot_B v.$ Then  since  $Tx$ is left symmetric we get $ v \bot_B Tx.$ Let $ \| x + y \| = r$, by orthogonality and strict convexity $ 1 < r <2.$
	Choose $0<t<1$ such that $ (1-t) (1 + \|T\|) < \frac{2-r}{1+2r}.$ Choose $\epsilon \in (0,1)$ such that $ (1-t) (1 + \|T\|) < \epsilon <\frac{2-r}{1+2r}.$ Clearly, $\frac{2-r}{1+2r}<1,$ since $1<r<2.$ Let $ w = (1-t)Tx + tv.$ Then $ \| w-v\| = (1-t) \|Tx-v\| \leq (1-t) (1 + \|T\|)<\epsilon.$ Now, any element $z\in \mathbb{X}$ can be written as $ z = ax+by +h$ where $a,b $ are scalars and $ h \in H_y.$ Consider a linear operator $ A : \mathbb{X} \longrightarrow \mathbb{X} $ defined as $ Az = av + bw.$ Then clearly, $A$ is compact. Clearly, $ T \bot_B A$ since  $ x \in M_T$ and $Tx \bot_B Ax.$ We next show that $A\not\perp_B T.$   Now,
	\begin{eqnarray*}
		\| A\big(\frac{x+y}{\|x + y\|}\big) \| & = & \frac{ \| v + w \|}{r} \\
		& = & \frac{\| 2v + w - v\|}{r} \\
		& \geq & \frac{2 - \epsilon}{r}\\
		& > & 1 + 2 \epsilon.
	\end{eqnarray*}
	Therefore, $ \|A\| > 1 + 2 \epsilon.$ It is easy to observe that $ x,y \notin M_A.$ Also, for each $ h \in H_y$ we get $ h \notin M_A$.
	Next, we claim that $ z \not\in M_A$ if $ab<0.$ Let $ z = -ax+by+h \in S_{\mathbb{X}}$ where $a>0,b>0, b-a >0.$ Then by using orthogonality we have, $ 1 = \| z\| = \| -ax+by+h\| > \mid a \mid, $ $ 1 = \|z\| = \| by + h - ax \| \geq \mid b \mid - \mid a \mid = b-a. $ Also $ \mid b \mid = \| by\| \leq \| by + h \| = \| z + ax \| \leq 2.$ Now,
	\begin{eqnarray*}
		\|Az\| & = & \| -av +bw\|\\
		& = & \| (b-a)v + b(w-v)\| \\
		& \leq & \mid b-a \mid + \mid b \mid \|w-v\| \\
		& < & \mid b - a \mid + \mid b \mid \epsilon \\
		& = & b-a + b \epsilon \\
		& \leq & 1 + 2 \epsilon \\
		& < & \| A(\frac{x+y}{\|x + y\|})\|.
	\end{eqnarray*}
	Next, we consider $ z = -ax+by+h \in S_{\mathbb{X}}$ where $a>0,b>0, b-a \leq0$.
	Then also $ \|Az\| < \mid b-a \mid + \mid b \mid \epsilon < a < 1 $. This shows that if $ z = -ax+by+h \in S_{\mathbb{X}},$ where $a>0,b>0 $  then $ z \not\in M_A.$ Similarly, considering $ z = ax-by+h \in S_{\mathbb{X}},$ where $a>0,b>0 $ we can show that $ z \notin M_A.$ So if $ z = ax+by+h \in M_A$ then we must have $ab>0.$ Next, our claim is that $Tz \not\in (Az)^{-}$ for all $ z \in M_A$. Let $ z = ax+by+h \in M_A.$ Then $ab>0, $  $Az= av + bw$  and  $Tz=aTx.$ Therefore, using Lemma \ref{lemma:minus}, we get $Tz \not\in (Az)^{-}.$ Now, using Theorem \ref{th-Kadets-Klee}, we get $A \not\perp_B T.$ This shows that $T$ is not left symmetric, a contradiction. This completes the proof of the theorem.
\end{proof}

In the following theorem, we study left symmetric compact operators. We first note that every compact operator on a reflexive Banach space attains its norm and so the norm attainment set is non-empty. Then using similar arguments as in Theorem \ref{Th-t=0} and \cite[Th. 2.1]{SPM} and \cite[Th. 2.1]{PSG} we can prove the following theorem.

\begin{theorem} Let $\mathbb{X}$ be reflexive, strictly convex Banach space and $\mathbb{Y}$ be a strictly convex Banach space. Then $ T \in K(\mathbb{X}, \mathbb{Y}) $ is left symmetric if and only if $T$ is the zero operator.
\end{theorem}

Now, the natural question that arises is that whether there is any non-zero left symmetric operator in any space. To answer this question, we first prove the following theorems.
\begin{theorem}\label{th-1121}
Let $\mathbb{X},\mathbb{Y}$ be normed linear spaces.
Suppose that
$\mathbb{X}$ is strictly convex and reflexive.
Suppose that $T\in K(\mathbb{X},\mathbb{Y})$ is left symmetric with $\|T\|=1$. If
$x_1\in M_T$, then $\dim T(\mathbb{X})=1$.
\end{theorem}
\begin{proof}
Fix $x_1\in M_T$.
Assume, contrary to
our claim, that $\dim T(\mathbb{X})>1$. Then there is $z\in S_\mathbb{X}$ such that $\dim\,{\rm span}\{Tz,Tx_1\}=2$.
It follows that there is $y\in {\rm span}\{z,x_1\}$ such that $y\perp_B x_1$, $y\neq 0$, so $Ty\neq 0$,
a contradiction (see Theorem \ref{th-1111}).
\end{proof}
 The following theorem is a reformulation of Theorem \ref{th-1121}.
\begin{theorem}\label{th-1131}
Let $\mathbb{X},\mathbb{Y}$ be normed linear spaces. Suppose that
$\mathbb{X}$ is strictly convex and reflexive. Let $T\in K(\mathbb{X},\mathbb{Y})$ be left
symmetric with $\|T\|=1$.
Then there are $w\in S_\mathbb{Y}$
and $f\in S_{\mathbb{X}^*}$ such that $T(\cdot)=f(\cdot)w$.
\end{theorem}

Observe that Theorem \ref{th-1131} can be strengthened as follows.
\begin{theorem}\label{th-1141}
Let $\mathbb{X},\mathbb{Y}$ be normed linear spaces. Suppose that
$\mathbb{X}$ is strictly convex and reflexive. Let $T\in K(\mathbb{X},\mathbb{Y})$ be left
symmetric with $\|T\|=1$. Assume that $x_1\in M_T$.
Then there are $w\in S_\mathbb{Y}$
and $f\in S_{\mathbb{X}^*}$ such that $T(\cdot)=f(\cdot)w$ (moreover, ${\rm card}M_T=2$ by strictly
convexity of $\mathbb{X}$, so we may assume $M_T=\{x_1,-x_1\}$
for some $x_1\in S_\mathbb{X}$).
Then:

{\rm (a)}\ $x_1$ is right symmetric, $w$ is left symmetric,

{\rm (b)}\  $x_1$ is left symmetric $\Leftrightarrow$ $x_1$ is smooth.
\end{theorem}
\begin{proof}
First we prove (a). Fix $y\in \mathbb{X}\setminus\{0\}$ such that $y\perp_B x_1$.
By Theorem \ref{th-1111}, we have $Ty=0$. Thus $f(y)=0$. Clearly, $|f(x_1)|=1.$ Therefore, by \cite[Th. 2.1]{J}, $x_1\perp_B y$. Hence, $x_1$ is right symmetric.

Now, fix $z\in \mathbb{Y}\setminus\{0\}$ such that $w\perp_B z$.
Define $A\in K(\mathbb{X},\mathbb{Y})$ by $A(\cdot):=f(\cdot)z$. Clearly, $M_T=M_A.$ It
is easy to check that $Tx_1\perp_B Ax_1$. So $T\perp_B A$.
It follows that $A\perp_B T.$ So, by \cite[Th. 2.1]{PSG}, $Ax_1\perp_B Tx_1$, hence $z\perp_B w$. Thus, $w$ is left symmetric.

Now we prove (b). If $x_1$ is left symmetric then by Proposition \ref{prop-symm} (iii) it follows that $x_1$ is smooth.  On the other hand if $x_1$ is smooth then by Proposition \ref{prop-symm} (i) it follows that $x_1$ is left symmetric.
\end{proof}

As a consequence of Theorem \ref{th-1141} and properties of the
spaces $Y,Y^*$, we will prove the Theorem \ref{th-1151}. The following lemma will be needed.
\begin{lemma}\label{lem-Y-A-y-K-f}
Let $X,Y$ be Banach space. Let $f\in \mathbb{X}^*$, $\|f\|=1$. The
spaces $\mathbb{Y}$ and $K_f(\mathbb{X},\mathbb{Y}):=\{A_y\!\in\! K(\mathbb{X},\mathbb{Y}): A_y(\cdot)\!:=\!f(\cdot)y,\  y\!\in\! \mathbb{Y}\}$
are isometrically
isomorphic.
\end{lemma}
The proof is very easy: define $\gamma \colon \mathbb{Y}\to K_f(\mathbb{X},\mathbb{Y})$ by the formula
$\gamma(y):=A_y$. The rest is clear. It follows from Lemma \ref{lem-Y-A-y-K-f} that
\begin{equation}\label{Az-perpB-Aw-z-perpB-w}
A_z\perp_B A_w\quad\Leftrightarrow\quad z\perp_B w.
\end{equation}
Now we are in position to prove another main result of this section.
\begin{theorem}\label{th-1151}
Let $\mathbb{X},\mathbb{Y}$ be reflexive Banach spaces. Suppose that
$\mathbb{Y}$ is smooth. Let $T\in K(\mathbb{X},\mathbb{Y})$ be left
symmetric with $\|T\|=1$.
Then there are $w\in S_\mathbb{Y}$
and $f\in S_{\mathbb{X}^*}$ such that $T(\cdot)=f(\cdot)w$ and $w$ is left symmetric.
\end{theorem}
\begin{proof}
It follows that $Y^*$ is strictly convex and $T^*\in K(\mathbb{Y}^*,\mathbb{X}^*)$ is left symmetric.
By the reflexivity of $Y^*$ and by a compactness of $T^*$ we get $M_{T^*}\neq\emptyset$.
It follows from Theorem \ref{th-1121} that $\dim T^*(\mathbb{Y}^*)=1$ and hence $\dim T(\mathbb{X})=1$
by a reflexivity of $\mathbb{X},\mathbb{Y}$.
Therefore there are $w\in S_\mathbb{Y}$
and $f\in S_{\mathbb{X}^*}$ such that $T(\cdot)=f(\cdot)w$.

We will show that $w$ is left symmetric.
Fix $z\in \mathbb{Y}\setminus\{0\}$ such that $w\perp_B z$.
Define $A_z\in K(\mathbb{X},\mathbb{Y})$ by $A_z(\cdot):=f(\cdot)z$.
Similarly, by the reflexivity of $\mathbb{X}$ and by a compactness of $T$ we get $M_{T}\neq\emptyset$.
Then for some $x_1\in M_T$ we have $Tx_1\perp_B A_zx_1$. So by Theorem \ref{th-Kadets-Klee} we have $T\perp_B A_z$.
It follows that $A_z\perp_B T$. It is easy to note that $T=A_w$.
Applying \eqref{Az-perpB-Aw-z-perpB-w} and Lemma \ref{lem-Y-A-y-K-f} we get $z\perp_B w$.
This means that $w$ is left symmetric.
\end{proof}
Now, we are in a position to exhibit non-zero left symmetric operators. Let $X\oplus_1 Z$ denote the space $X\times Z$ with $\|(x,z)\|_1:=\|x\|+\|z\|$. We characterize left
symmetric linear operators from $X\oplus_1\mathbb{R}$ into a reflexive smooth space $Y$.
\begin{theorem}\label{th-direct-sum-1111}
Let $\mathbb{X}$ be reflexive Banach space. Suppose that $\mathbb{Y}$ is reflexive and smooth Banach space.
Let $T\in K(\mathbb{X}\oplus_1 \mathbb{R},\mathbb{Y})$ be a nonzero operator with $\|T\|=1=\|T(0,1)\|$.
The following conditions are equivalent:

{\rm (a)}\ $T$ is left symmetric,

{\rm (b)}\  there are $w\in S_\mathbb{Y}$, $f\in S_{(\mathbb{X}\oplus_1\mathbb{R})^*}$ such
that $T(\cdot)=f(\cdot)w$, $\mathbb{X}=\ker f$ and $w$ is left symmetric.
\end{theorem}
\begin{proof}
We will prove (a)$\Rightarrow$(b). It follows from Theorem \ref{th-1151} that
there are $w\in S_\mathbb{Y}$, $f\in S_{(\mathbb{X}\oplus_1\mathbb{R})^*}$ such
that $T(\cdot)=f(\cdot)w$, and $w$ is left symmetric.

In order to prove $\mathbb{X}=\ker f$, we note $\mathbb{X}\perp_B (0,1)$.
It follows from Theorem \ref{th-1111} that $T(\mathbb{X})=\{0\}$, therefore $\mathbb{X}\subset \ker f$.
Since ${\rm co}\dim \mathbb{X}=1$, hence $\mathbb{X}=\ker f$.

In order to prove (b)$\Rightarrow$(a), assume that
there are $w\in S_\mathbb{Y}$, $f\in S_{(\mathbb{X}\oplus_1\mathbb{R})^*}$ such
that $T(\cdot)=f(\cdot)w$, $\mathbb{X}=\ker f$ and $w$ is left symmetric. Suppose that $T\perp_B A$ and $A\neq 0$. It is easy to see that
$M_T=D\cup-D$ for some connected closed subset $D\subset S_\mathbb{X}$.
So, there is $x_1\in M_T$ such that $Tx_1\perp_B Ax_1$ and $Tx_1=w$,
whence $w\perp_B Ax_1$. Since $w$ is left symmetric,
we conclude that $Ax_1\perp_B w$, so we get $Ax_1\perp_B Tx_1$.
Let $x_2\in M_A$. If $x_2\in\{x_1,-x_1\}$, then we have $Ax_1\perp_B Tx_1$, so $A\perp_B T$.
But, if $x_2\notin\{x_1,-x_1\}$ (more precisely, $x_2\!\notin\!\{(0,1),-(0,1)\}$) then
there exists an extreme point
$(e,0)\in{\rm Ext} B_{\mathbb{X}\oplus_1 \mathbb{R}}$ such that $\|A(e,0)\|=\|A\|$. Since
$\mathbb{X}=\ker f$, we have $T(e,0)=0$. It yields $A(e,0)\perp_B T(e,0)$, so $A\perp_B T$.
The proof is completed.
\end{proof}

Using theorem \ref{th-direct-sum-1111}, we now characterize the left symmetric operators from $\ell_1^n$ to a reflexive, smooth Banach space. Let $e_1=(1,0,\ldots,0)$, $e_2=(0,1,0\ldots,0)$, $\ldots$, $e_n=(0,\ldots,0,1)$ denote the extreme points of the closed unit ball $B_{l^n_1}$.
\begin{theorem}\label{th-direct-sum-1111-l}
Suppose that $\mathbb{Y}$ is reflexive and smooth Banach space.
Let $T\in K(l^n_1,\mathbb{Y})$ be a nonzero operator with $\|T\|=1$, $n\geq 2$.
The following conditions are equivalent:

{\rm (a)}\ $T$ is left symmetric,

{\rm (b)}\  there are $w\in S_\mathbb{Y}$, $f\in S_{(\ell_1^n)^*}$, $e_k\in\{e_1,\ldots,e_n\}$ such
that $T(\cdot)=f(\cdot)w$, and $w$ is left
symmetric and $|f(e_k)|=1$ and $f(e_j)=0$ for $e_j\in\{e_1,\ldots,e_n\}\setminus\{e_k\}$.
\end{theorem}
\begin{proof}
Since $\|T\|=1$, there exists $e_k\in \{e_1,\ldots,e_n\}$ such that $e_k\in M_T$. Without any loss of generality
we can assume that $e_k=e_n$, i.e., $k=n$. Then we may write $l^n_1=l^{n-1}_1\oplus_1\mathbb{R}$.
Next we apply Theorem \ref{th-direct-sum-1111} and the proof is completed.
\end{proof}

\section{ Right symmetric operators}

We begin this section with a simple but important observation in the form of following theorem.

\begin{theorem}\label{th-T-smooth-not-right-sym}
Let $\mathbb{X}$ be reflexive, strictly convex Banach space and $\mathbb{Y}$ be a strictly convex Banach space. If $ T \in K(\mathbb{X}, \mathbb{Y}) $ is smooth then  $T$ cannot be right symmetric.
\end{theorem}
\begin{proof}
If possible, let $T$ be right symmetric. Then by Proposition \ref{prop-symm} it follows that $T$ is left symmetric. Then $T$ must be zero operator, which is not possible as $T$ is smooth. Thus $T$ cannot be right symmetric.
\end{proof}
\begin{remark} This not only improves on \cite[Th. 2.3]{SGP} but also gives an elegant simple proof of the same.
\end{remark}

In the following proposition, we study the properties of $Tx,$ where $T$ is a right symmetric operator and $M_T=\{\pm x\}.$

\begin{prop}\label{prop:rightsymmetric}
Let $T$ be a compact linear operator on a reflexive Banach space and $ M_T = \{ \pm x \}.$ If $T$ is right symmetric then $Tx$ is  right symmetric. Moreover, if $Tx$ is smooth, then $Tx$ is also left symmetric.
\end{prop}
\begin{proof}
If possible, let there exists $ y \in S_{\mathbb{X}}$ such that $ y \bot_B Tx $ but $ Tx \not\perp_B y.$   Since $x \in M_T,$ so by \cite[Lemma 2.1]{SPM2}, there exists a hyperspace $H_x$ such that $ x \bot_B H_x$ and $Tx \bot_B T(H_x)$.  Consider a linear operator $ A : \mathbb{X} \rightarrow \mathbb{X}$ such that $ A( a x + h) = a y$ where $a$ is  a scalar and $h \in H_x.$ Then it is easy to observe that $A$ is compact. Clearly, $x \in M_A$ and $ Ax \bot_B Tx $ which implies that $ A \bot_B T.$ But  $ M_T = \{ \pm x\} $ and $ Tx \not\perp_B Ax$, so using Theorem 2.1 of \cite{PSG}, we conclude that $ T \not\perp_B A.$ This contradicts the fact that $T$ is right symmetric. Thus, $Tx$ must be right symmetric.\\
If $Tx$ is smooth, then by  Proposition \ref{prop-symm} (i) it follows that $Tx$ is also left symmetric.
\end{proof}

Now, we are ready to prove (under some assumption) that every right symmetric operator must be extreme point.
\begin{theorem}\label{th-dim2-r-sym-exp}
Let $\mathbb{X},\mathbb{Y}$ be striclty convex Banach spaces.
Suppose that $\dim \mathbb{X}=2$. Assume that $Y$ is smooth. Let $T\in K(\mathbb{X},\mathbb{Y})$, $\|T\|=1$.
If $T$ is right symmetric, then $T$ is extreme point
of the closed unit ball of $K(\mathbb{X},\mathbb{Y})$.
\end{theorem}
\begin{proof}
Assume that $T\in K(\mathbb{X},\mathbb{Y})$ is right symmetric with $\|T\|=1$.
We show that ${\rm card} M_T\geq 4$.
Assume, contrary to
our claim, that ${\rm card} M_T=2$. Then using \cite[Th. 4.2]{PSG}, we get $T$ is smooth.
It follows from Theorem \ref{th-T-smooth-not-right-sym} that $T$ is not right symmetric, a contradiction.

Thus we have shown that ${\rm card} M_T\geq 4$. Therefore there are $a,b\in M_T$ such
that $a,b$ are linearly independent. Next we are going to prove that $T\in {\rm Ext}B_{K(\mathbb{X},\mathbb{Y})}$.
Assume that $T=\lambda U+(1-\lambda)W$ for some $U,W\in B_{K(\mathbb{X},\mathbb{Y})}$ and for some $\lambda\in(0,1)$.
It follows that $Ta=\lambda Ua+(1-\lambda)Wa$ and $Tb=\lambda Ub+(1-\lambda)Wb$. In particular, we have
$Ua,Wa,Ub,Wb\in B_{\mathbb{Y}}$. Since $\mathbb{Y}$ is strictly convex, we
obtain $Ta=Ua=Wa$ and $Tb=Ub=Wb$.
We have shown that $T,U$ and $W$ coincides on the basis $\{a,b\}$, thus they are equal: $T=U=W$.
That means $T\in {\rm Ext}B_{K(\mathbb{X},\mathbb{Y})}$.
\end{proof}
%\begin{remark}
%Since $T$ is smooth and $ M_T = \{ \pm x\} $ so  $Tx$ is also smooth. Then by Proposition \ref{prop-symm} (i) it follows that $Tx$ is also left symmetric.
%\end{remark}

\noindent When $ \mathbb{X} $ is not necessarily strictly convex or smooth, we have the following two theorems regarding right symmetric operators.
\begin{theorem}
Let $T$ be a non-zero compact operator on a reflexive, Kadets-Klee Banach space $\mathbb{X}$. If $T$ is smooth and  $\|T\|$ is a spectral value of $T,$ then either nullity $T = 0$ or $T$ is not right symmetric.
\end{theorem}
\begin{proof} Since $T$ is compact and smooth on a reflexive Banach space so from \cite[Th. 4.2]{PSG}, we get $ M_T = \{ \pm x \}$ for some $x \in S_{\mathbb{X}}$ and $\|Tx\| = \|T\|. $ Since every non-zero spectral value of a compact operator is an eigenvalue of $T,$ so  we get $Tx =  \|T\|x$ or $ Tx = - \|T\|x.$ If nullity $T = 0,$ then we are done. So assume  nullity $T \geq 1.$  Let  $ u \in S_{\mathbb{X}} $ be such that $ Tu = 0.$ Then for any scalar  $\lambda,~\|I+\lambda T\|\geq \|(I+\lambda T)u\|=\|Iu\|=\|I\|.$ Thus, $I\perp_B T.$ If possible, suppose that $T\perp_B I.$ Then by Theorem \ref{th-Kadets-Klee}, $Tx\perp_B Ix.$ It follows that $\pm \|T\|x\perp_B x,$ a contradiction. Therefore, $T$ is not right symmetric.
\end{proof}

\begin{theorem}
Let $\mathbb{X,\mathbb{Y}}$ be two normed linear spaces and $T\in B(\mathbb{X},\mathbb{Y})$ be such that kernel of $T$ contains a non-zero point. Then either $T \bot_B I $ or $T$ is not right symmetric.
\end{theorem}
\begin{proof} Let $u\in S_{\mathbb{X}}$ be such that $Tu=0.$ Then $ \| I + \lambda T \| \geq \| (I+\lambda T)u \| = \|Iu\| = \|I\|,$ i.e., $ I \bot_B T.$  If $T\bot_B I$ then there is nothing to show, otherwise we get $ I \bot_B T $ but $T \not\perp_B I $ and so $T$ is not right symmetric.

\end{proof}
\begin{remark} We note that if  kernel $T$ contains a non-zero element then clearly $I \bot_B T.$ Then it is easy to see that the last theorem substantially  improves on \cite[Th. 2.5]{SGP}.

\end{remark}
\bibliographystyle{amsplain}

\end{document}